\documentclass[11pt]{amsart}
\usepackage{amsmath}
\usepackage{amssymb}
\usepackage{graphicx}
\usepackage{eurosym}
\usepackage{amsfonts}
\usepackage{xcolor}
\usepackage{mathtools}

\newtheorem{Theorem}{Theorem}[section]

\newtheorem{Lemma}[Theorem]{Lemma}
\newtheorem{Proposition}[Theorem]{Proposition}
\theoremstyle{definition}
\newtheorem{Definition}[Theorem]{Definition}
\newtheorem{Example}[Theorem]{Example}
\newtheorem{Remark}[Theorem]{Remark}

 \numberwithin{equation}{section}
\email{{r.precup@ictp.acad.ro}} \email{andrei.stan@ubbcluj.ro}

\begin{document}

\begin{abstract}
In this paper, we extend the concept of \( b \)-metric spaces to the vectorial case, where the distance is vector-valued, and the constant in the triangle inequality axiom is replaced by a matrix. For such spaces, we establish results analogous to those in the \( b \)-metric setting: fixed-point theorems, stability results, and a variant of Ekeland's variational principle. As a consequence, we also derive a variant of Caristi's fixed-point theorem.
\end{abstract}
\title[]{Fixed
point results and the Ekeland variational principle in vector $B$-metric
spaces}
\author{Radu Precup}
\address{Faculty of Mathematics and Computer Science and Institute of Advanced Studies in Science and Technology, Babe\c{s}-Bolyai University, 400084 Cluj-Napoca, Romania}
\address{Tiberiu Popoviciu Institute of Numerical Analysis, Romanian Academy, P.O. Box 68-1, 400110 Cluj-Napoca, Romania}
\email{r.precup@ictp.acad.ro}

\author{Andrei Stan}
\address{Department of Mathematics, Babe\c{s}-Bolyai University, 400084 Cluj-Napoca, Romania}
\email{andrei.stan@ubbcluj.ro}
% \author[A. Stan]{Andrei Stan}
% \address[A. Stan]{A. Stan, Department of Mathematics, Babe\c{s}-Bolyai
% University, 400084 Cluj-Napoca, Romania \& Tiberiu Popoviciu Institute of
% Numerical Analysis, Romanian Academy, P.O. Box 68-1, 400110 Cluj-Napoca,
% Romania}

\maketitle

\section{Introduction}

% The concept of $b$-metric spaces  comes  from the natural generalization of the properties of a metric, by allowing the right hand side of triangle inequality to be multiplied with a constant $b\geq 1$. Althougohot ideas in this directions can be found in \cite{hyres}, it is accepted that the proper terminology of $b$-metric spaces is given in \cite{bourbaki} or \cite{bahtin}.

The concept of a $b $-metric space arises as a natural generalization of a
metric space, where the triangle inequality axiom is relaxed by introducing
a constant $b \geq 1 $ on its right-hand side. Early ideas in this direction
can be traced back to the notion of "quasimetric" spaces, as discussed in
\cite{hyres}. However, the formal definition and terminology of $b $-metric
spaces are widely attributed to Bakhtin \cite{bahtin} and Czerwik \cite%
{Czerwik1993}. Notably, one of the earliest works to introduce a mapping
satisfying the properties of a $b $-metric dates back to 1970 in \cite%
{coifman}, where such a mapping was referred to as a "distance". A concept
related to that of a $b $-metric is the notion of a quasi-norm, which can be
traced back to Hyers \cite{hyres2} and Bourgin \cite{Bourgin}, who
originally used the term "quasi-norm." For a survey on $b$-metric spaces we
send the reader to \cite{berinde,survey}.

% Various results from classical theorey of metric spaces has been adapted to $b$-metric space: fixed point theorems, stability results, well posedness or variational principles (see, e.g.,\cite{MITROVIC, berinde2,bahtin}). In \cite{monica}, the metric was allowes to be vector values and results similar to $b$-metric spaces has been provided, with matrices convegent to zero instead of contraction constants.
Various results from the classical theory of metric spaces have been
extended to $b $-metric spaces, including fixed-point theorems (see, e.g.,
\cite{MITROVIC, bpr,fixed point,Reich2001,kirk carte,berinde2}), estimations
(see, e.g.,\cite{Miculescu, Suzuki}), stability results (see, e.g, \cite%
{stability monica,pp}), and variational principles (see, e.g., \cite%
{bmv,ekeland2}). In \cite{monica}, the metric was allowed to take vector
values, and results analogous to those for $b $-metric spaces were
established, with matrices converging to zero replacing the contraction
constants, but not the constant $b$ from the triangle inequality axiom.

% In the present paper, we introduce the notion of vectorial $B$-metric spaces by simply replacing the constant $b$ from triangle inequality with a matrix. This change arises some difficulties in establishing results similar to $b$-metric spaces. To our best knowledge, this concept is new, as well as all the related results.

In this paper, we introduce the concept of a vector $B $-metric space, where
the scalar constant $b $ in the triangle inequality is replaced by a matrix $%
B $. This generalization introduces new challenges in establishing results
analogous to those for classical $b $-metric spaces. To the best of our
knowledge, this concept, along with the corresponding results presented
here, is novel. Notably, some of the results appear to be new even in the
scalar particular case where the matrix $B $ is reduced to a constant.

Throughout this paper, we consider $\mathbb{R}^n$-valued vector metrics ($n
\geq 1 $) on a set $X $, i.e., mappings $d: X \times X \to \mathbb{R}_+^n $.
In the scalar case ($n = 1 $), we use the special notation $\rho $ to denote
a standard metric or a $b$-metric.

% We begin this section with basic notions and results regarding $b$-metrics that will be used throughout this paper.
The classical definition of a $b$-metric reads as follows:

\begin{Definition}
Let $X$ be a set and let $b \geq 1$ be a given real number. A mapping $\rho
: X \times X \to \mathbb{R}_+$ is said to be a \textit{b-metric} if for all $%
x, y, z \in X$ the following conditions are satisfied: $\rho(x, y)\geq 0$, $%
\rho(x, y) = 0$ if and only if $x = y$, $\rho(x, y) = \rho(y, x)$ and $%
\rho(x, z) \leq b \left(\rho(x, y) +\rho(y, z)\right)$. The pair $(X, \rho)$
is called a $b$-metric space.
\end{Definition}

In case the mapping $\rho$ is allowed to be vector-valued and one replaces
the constant $b $ by a matrix $B $, we obtain our definition of a vector $B$%
-metric space.

\begin{Definition}
\label{definitie B metrica vectiriala}  Let $X$ be a set, $n\geq 1$ and let $%
B\in \mathcal{M}_{n\times n}(\mathbb{R})$ be an arbitrary matrix.  A mapping
$d=(d_1,d_2, \ldots, d_n)\colon X\times X\to \mathbb{R}^n_+$ is called a
vector $B$-metric if for all $u,v,w\in X,$ one has

\begin{enumerate}
\item[ ] \text{(positivity)}: $d(u, v) \geq 0 $ and $d(u, v) = 0 $ if and
only if $u = v $;

\item[ ]  \text{(symmetry)}: $d(u, v) = d(v, u) $;

\item[ ] \textit{(triangle inequality)}: $d(u, w) \leq B \left(d(u, v) +
d(v, w)\right) $.
\end{enumerate}
The pair $(X,d)$ is called a vector $B$-metric space.
\end{Definition}

\section{ Preliminaries}

In this paper, the vectors in $%
%TCIMACRO{\U{211d} }%
%BeginExpansion
\mathbb{R}
%EndExpansion
^{n}$ are looked as column matrices and ordering between them and, more
generally, between matrices of the same size is understood by components.
Likewise, the convergence of a sequence of vectors or matrices is understood
componentwise.

The spaces of square matrices of size $n$ with real number entries and
nonnegative entries are denoted by $\mathcal{M}_{n\times n}\left(
%TCIMACRO{\U{211d} }%
%BeginExpansion
\mathbb{R}
%EndExpansion
\right) $ and $\mathcal{M}_{n\times n}\left(
%TCIMACRO{\U{211d} }%
%BeginExpansion
\mathbb{R}
%EndExpansion
_{+}\right) ,$ respectively. An element of $\mathcal{M}_{n\times n}\left(
%TCIMACRO{\U{211d} }%
%BeginExpansion
\mathbb{R}
%EndExpansion
_{+}\right) $ is refereed as a \textit{positive matrix}, while a matrix $%
M\in \mathcal{M}_{n\times n}\left(
%TCIMACRO{\U{211d} }%
%BeginExpansion
\mathbb{R}
%EndExpansion
\right) $ is called \textit{inverse-positive} if it is invertible and its
inverse $M^{-1}$ is positive.

A positive matrix $M$ is said to be \textit{convergent to zero} if its power
$M^{k}$ tends to the zero matrix $0_{n}$ as $k\rightarrow \infty .$

One has the following characterizations of matrices which are convergent to
zero (see, e.g., \cite{p role,nonnegative}).

\begin{Proposition}
Let $M\in \mathcal{M}_{n\times n}\left(
%TCIMACRO{\U{211d} }%
%BeginExpansion
\mathbb{R}
%EndExpansion
_{+}\right) $ and let $I$ be the identity matrix of size $n.$ The following
statements are equivalent:
\begin{description}
\item[(a)] $M$ is convergent to zero.

\item[(b)] The spectral radius $r\left( M\right) $ of matrix $M$ is less
than $1,$ i.e., $r\left( M\right) <1.$

\item[(c)] $I-M$ is invertible and $\left( I-M\right) ^{-1}=I+M+M^{2}+\ ....$

\item[(d)] $I-M$ is inverse-positive.
\end{description}
\end{Proposition}

The following proposition collects the various properties equivalent to the
notion of an inverse-positive matrix (see, e.g., \cite{nonnegative,collatz}).

\begin{Proposition}
Let $M\in \mathcal{M}_{n\times n}\left(
%TCIMACRO{\U{211d} }%
%BeginExpansion
\mathbb{R}
%EndExpansion
\right) .$ The following statements are equivalent:
\end{Proposition}

\begin{description}
\item[(a)] $M$ is inverse-positive.

\item[(b)] $M$ is monotone, i.e., $Mx\geq 0$ $\left( x\in
%TCIMACRO{\U{211d} }%
%BeginExpansion
\mathbb{R}
%EndExpansion
^{n}\right) $ implies $x\geq 0.$

\item[(c)] There exists a positive matrix $\overline{M}$ and a real number $%
s>r\left( \overline{M}\right) $ such that the following representation
holds: $M=sI-\overline{M}.$
\end{description}

Clearly, if $M$ is inverse-positive, from the representation $M=sI-\overline{%
M},$ we immediately see that all its entries except those from the diagonal
are $\leq 0;$ also the matrix $\frac{1}{s}\overline{M}$ is convergent to
zero. If a matrix $M$ is both positive and inverse-positive, using the
representation $M=sI-\overline{M}$ we deduce that $M$ must be a diagonal
matrix with strictly positive diagonal entries.

A mapping $N:X\rightarrow X$ defined on a vector $B$-metric space $\left(
X,d\right) $ is said to be a \textit{Perov contraction mapping} if there
exists a matrix $A$ convergent to zero such that%
\begin{equation}
d\left( N\left( x\right) ,N\left( y\right) \right) \leq Ad\left( x,y\right)
\label{f0}
\end{equation}%
for all $x,y\in X.$

The next proposition is about the relationship between vector $B$-metrics
and both vector and scalar $b$-metrics.

\begin{Proposition}
(1$^{0}$) Any vector-valued $b$-metric $d$ can be identified with a vector $%
B_{b}$-metric, where $B_{b}$ is the diagonal matrix whose diagonal entries
are all equal to $b.$

(2$^{0}$) If $d$ is a vector $B$-metric with an inverse-positive matrix $B,$
then $d$ is also a vector $\underline{B}$-metric with respect to the
diagonal matrix $\underline{B}$ that preserves the diagonal of $B,$ as well
as a vector-valued $\Tilde{b}$-metric with $\Tilde{b}=\max \left\{ b_{ii}:\
1\leq i\leq n\right\} .$ Here $B=\left( b_{ij}\right) _{1\leq i,j\leq n}.$

(3$^{0}$) If $d$ is a vector $B$-metric with a positive matrix $B,$ then to
each norm in $%
%TCIMACRO{\U{211d} }%
%BeginExpansion
\mathbb{R}
%EndExpansion
^{n}$ one can associate a scalar $b$-metric, for example:%
% \begin{equation*}
% \begin{array}{lll}
% \rho _{1}\left( x,y\right) :=\sum\limits_{i=1}^{n}d_{i}\left( x,y\right) &
% \text{is a }b_{1}\text{-metric,} & b_{1}:=\sum\limits_{i=1}^{n}\max_{1\leq
% j\leq n}b_{ij} \\
% \rho _{\infty}\left( x,y\right) :=\max_{1\leq i\leq n}d_{i}\left( x,y\right) &
% \text{is a }b_{\infty}\text{-metric,} & b_{\infty}:=\max_{1\leq i\leq
% n}\sum\limits_{j=1}^{n}b_{ij} \\
% \rho _{2}\left( x,y\right) :=\left( \sum\limits_{i=1}^{n}d_{i}\left(
% x,y\right) ^{2}\right) ^{\frac{1}{2}} & \text{is a }b_{2}\text{-metric,} &
% b_{2}:=\left( \sum\limits_{i,j=1}^{n}b_{ij}^{2}\right) ^{\frac{1}{2}}.%
% \end{array}%
% \end{equation*}

\begin{equation*}
\begin{aligned} \rho_1(x, y) &:= \sum\limits_{i=1}^{n} d_i(x, y), & \text{is
a } b_1\text{-metric,} & \quad b_1 := \sum\limits_{i=1}^{n} \max_{1 \leq j
\leq n} b_{ij}, \\ \rho_{\infty}(x, y) &:= \max_{1 \leq i \leq n} d_i(x, y),
& \text{is a } b_{\infty}\text{-metric,} & \quad b_{\infty} := \max_{1 \leq
i \leq n} \sum\limits_{j=1}^{n} b_{ij}, \\ \rho_2(x, y) &:= \left(
\sum\limits_{i=1}^{n} d_i(x, y)^2 \right)^{\frac{1}{2}}, & \text{is a }
b_2\text{-metric,} & \quad b_2 := \left( \sum\limits_{i,j=1}^{n} b_{ij}^2
\right)^{\frac{1}{2}}. \end{aligned}
\end{equation*}
\end{Proposition}

Thus, to any vector $B$-metric, one can associate different (scalar) $b$%
-metrics, depending on the chosen metric on $%
%TCIMACRO{\U{211d} }%
%BeginExpansion
\mathbb{R}
%EndExpansion
^{n}.$ However, as shown in \cite{p role}, working in a vector setting with
matrices instead of numbers is more accurate especially when a connection
with other matrices is necessary. It will also be the case of this work
where some conditions or conclusions will connect the matrix $B$ with the
matrix $A$ involved in (\ref{f0}).

If $Y$ is a nonempty subset of a vector $B$-metric space $\left( X,d\right)
, $ we define the \textit{diameter} of the set $Y$ by
\begin{equation*}
\mathrm{diam}_{d}(Y):=\sup \{\rho _{1}\left( x,y\right) \,:\ x,y\in Y\}=\sup
\left\{\sum\limits_{i=1}^{n}d_{i}\left( x,y\right) \,:\ x,y\in Y\ \right\}.
\end{equation*}%
From this definition, it follows immediately that if $\ \mathrm{diam}%
_{d}\left( Y\right) =a,$ then $d\left( x,y\right) \leq ae$ $\ $for all $%
x,y\in Y,$ where $e=(1,1,\ldots ,1)\in \mathbb{R}^{n}$. Conversely, if $%
d\left( x,y\right) \leq ae$ for all $x,y\in Y,$ then $\mathrm{diam}%
_{d}(Y)\leq na.$

Although a $b$-metric does not generate a topology (see, e.g., \cite{cobzas}%
), several topological properties can still be defined in terms of sequences
(e.g., closed sets, continuous operators, or lower semicontinuous
functionals).

We conclude this section by two examples of vector $B$-metrics.

\begin{Example}
Let $\ d\colon \mathbb{R}^{2}\times \mathbb{R}^{2}\rightarrow \mathbb{R}%
_{+}^{2}$\ \ be given by%
\begin{equation*}
d(x,y)=%
\begin{pmatrix}
|x_{1}-y_{1}|^{2}+|x_{2}-y_{2}| \\
|x_{2}-y_{2}|,%
\end{pmatrix}%
,
\end{equation*}%
for $x=(x_{1},x_{2}),\ y=(y_{1},y_{2})\in
%TCIMACRO{\U{211d} }%
%BeginExpansion
\mathbb{R}
%EndExpansion
^{2}.$ Then, $\left(
%TCIMACRO{\U{211d} }%
%BeginExpansion
\mathbb{R}
%EndExpansion
^{2},d\right) $ is a vector $B$-metric space, where
\begin{equation*}
B=%
\begin{pmatrix}
2 & -1 \\
0 & 1%
\end{pmatrix}%
.
\end{equation*}%
Here, the matrix $B$ is inverse-positive, but not positive.
\end{Example}

% \begin{Example}
% % We present an example of vector-valued mapping $d$ which is a vector $B$-metric
% % with respect to a positive matrix and for which there are no
% % inverse-positive matrices with respect to which it be a $B$-metric.Here’s a refined and polished version of your statement for clarity, correctness, and formal academic tone:
% \end{Example}

\begin{Example}
We present an example of a vector-valued mapping $d$ which is a vector $B$%
-metric with respect to a positive matrix, but for which no inverse-positive
matrix exists such that $d$ remains a vector $B$-metric. Let
\begin{equation*}
S=\left\{ (t,t)\,:\ \,t\in \mathbb{R}\right\} \subset \mathbb{R}^{2},
\end{equation*}%
and let $\ d\colon \mathbb{R}^{2}\times \mathbb{R}^{2}\rightarrow \mathbb{R}%
_{+}^{2}$\ \ be given by
\begin{equation*}
d(x,y)=%
\begin{cases}
(0,0) & \text{if }x=y, \\
\left( |x-y|^{2},|x-y|\right)  & \text{if }x,y\in S, \\
\left( |x-y|,|x-y|^{2}\right)  & \text{otherwise},%
\end{cases}%
\end{equation*}%
where $\left\vert z\right\vert =|(z_{1},z_{2})|=|z_{1}|+|z_{2}|$ is a norm
on $\mathbb{R}^{2}$. Note that $d$ is a vector $B_{0}$-metric, where
\begin{equation*}
B_{0}=%
\begin{pmatrix}
2 & 2 \\
1 & 1%
\end{pmatrix}%
.
\end{equation*}%
Let us show that $B_{0}$ is the smallest matrix for which the triangle
inequality holds for $d$. To this aim, let $B=(b_{ij})_{1\leq i,j\leq n}$ be
any matrix for which the triangle inequality is satisfied. Then, for $x,y\in
S$ and $z\notin S$, we have
\begin{equation}
\begin{pmatrix}
|x-y|^{2} \\
|x-y|%
\end{pmatrix}%
\leq
\begin{pmatrix}
b_{11}\left( |x-z|+|z-y|\right) +b_{12}\left( |x-z|^{2}+|z-y|^{2}\right)  \\
b_{21}\left( |x-z|+|z-y|\right) +b_{22}\left( |x-z|^{2}+|z-y|^{2}\right)
\end{pmatrix}%
.  \label{ineq generala exemplu}
\end{equation}%
Let $t,\alpha \in \mathbb{R}\setminus \left\{ 0\right\} $, and set $%
x=(t,t)\in S,$ $y=(0,0)\in S$ and $z=(\alpha ,0)\notin S$. The first
inequality in \eqref{ineq generala exemplu} yields,
\begin{equation*}
4t^{2}\leq b_{11}\left( |t-\alpha |+|t|+\left\vert \alpha \right\vert
\right) +b_{12}\left( (|t-\alpha |+|t|)^{2}+\alpha ^{2}\right) .
\end{equation*}%
Clearly, taking $\alpha =t$ and the limit as $t\rightarrow \infty $, this
inequality holds only if $b_{12}\geq 2$. Similarly, from the second
inequality, we obtain
\begin{equation*}
2\left\vert t\right\vert \leq b_{21}\left( |t-\alpha |+|t|+\left\vert \alpha
\right\vert \right) +b_{22}\left( (|t-\alpha |+|t|)^{2}+\alpha ^{2}\right) .
\end{equation*}%
Setting $\alpha =\frac{t}{2}$, we find that
\begin{equation*}
2\left\vert t\right\vert \leq 2b_{21}\left\vert t\right\vert +5b_{22}\frac{%
t^{2}}{2},\quad \text{or equivalently,}\quad 5b_{22}\frac{t^{2}}{2}%
+2\left\vert t\right\vert (b_{21}-1)\geq 0.
\end{equation*}%
Clearly, this inequality required for all $t$ implies $b_{21}\geq 1$. To
determine the values of $b_{11}$ and $b_{22}$, we apply the triangle
inequality with $x,y,z\in S\ \,(x\neq y\neq z\neq x)$, which gives
\begin{equation*}
\begin{pmatrix}
|x-y|^{2} \\
|x-y|%
\end{pmatrix}%
\leq
\begin{pmatrix}
b_{11}\left( |x-z|^{2}+|z-y|^{2}\right) +b_{12}\left( |x-z|+|z-y|\right)  \\
b_{21}\left( |x-z|^{2}+|z-y|^{2}\right) +b_{22}\left( |x-z|+|z-y|\right)
\end{pmatrix}%
.
\end{equation*}%
Similar arguments as above imply that $b_{11}\geq 2$ and $b_{22}\geq 1.$
Thus, $B\geq B_{0}$ as claimed.
\end{Example}

\section{Fixed point theorems in vector $B$-metric spaces}

In this section we establish some fixed point results in vector $B$-metric
spaces, analogous to the well-known classical results.

\subsection{Perov type fixed point theorem}

% The first result  is a Perov type fixed point theorem in such spaces.

Our first result is a version of Perov's fixed point theorem (see, \cite%
{perov1, perov2}) for such spaces.

\begin{Theorem}
\label{Perov} Let $(X,d)$ be a complete vector $B$-metric space, where $B$
is either a positive or an inverse-positive matrix, and let $N\colon
X\rightarrow X$ be an operator. Assume that there exists a convergent to
zero matrix $A\in \mathcal{M}_{n\times n}(\mathbb{R}_{+})$ such that
\begin{equation}
d(N(x),N(y))\leq Ad(x,y),\quad \text{for all }x,y\in X,
\label{conditie contractie matrice}
\end{equation}%
i.e., $N$ is a Perov contraction mapping. Then, $N$ has a unique fixed point.
\end{Theorem}

\begin{proof}
Let $x_{0}\in X$, and recursively define
\begin{equation*}
x_{k}=N(x_{k-1}),\ \ \ \text{for\ \ }k\geq 1.
\end{equation*}%
Since the matrix $A$ is convergent to zero, for each $\alpha >0$, there
exists $k_{0}=k_{0}\left( \alpha \right) $ such that $\ $%
\begin{equation*}
A^{k_{0}}\leq \Lambda ,
\end{equation*}%
where $\Lambda $ is the square matrix of size $n$ whose entries are all
equal to $\alpha .$ Let $k,p\geq 0$ and $k_{0}$ be such that $A^{k_{0}}\leq
\Lambda $, for some $\alpha >0$ to be specified later.

Case (a): $B$ is inverse-positive$.$ The triangle inequality yields
\begin{align*}
B^{-2}d(x_{k},x_{p})& \leq
B^{-1}d(x_{k},x_{k+k_{0}})+B^{-1}d(x_{p},x_{k+k_{0}}) \\
& \leq
B^{-1}A^{k}d(x_{0},x_{k_{0}})+d(x_{p},x_{p+k_{0}})+d(x_{p+k_{0}},x_{k+k_{0}})
\\
& \leq
B^{-1}A^{k}d(x_{0},x_{k_{0}})+A^{p}d(x_{0},x_{k_{0}})+A^{k_{0}}d(x_{k},x_{p})
\\
& \leq B^{-1}A^{k}d(x_{0},x_{k_{0}})+A^{p}d(x_{0},x_{k_{0}})+\Lambda
d(x_{k},x_{p}),
\end{align*}%
which gives
\begin{equation}
(B^{-2}-\Lambda )d(x_{k},x_{p})\leq
B^{-1}A^{k}d(x_{0},x_{k_{0}})+A^{p}d(x_{0},x_{k_{0}}).  \label{ineq perov}
\end{equation}%
%
%
%
%
%
%
% In what follows,  our aim to demonstrate that a linear combination of the elements of the vector \( d(x_k, x_p) \) is dominated by the terms on the right-hand side of \eqref{ineq1}, since they are convergent to zero as $k,p\to \infty
% $.
Given that the right-hand side of \eqref{ineq perov} is a vector that
converges to zero as $k,p\rightarrow \infty $, our goal is to show that a
linear combination of the components of the vector $d(x_{k},x_{p})$ is
bounded above by the corresponding components of the right-hand side of %
\eqref{ineq perov}. To this aim, we make the following notations
\begin{align*}
& B^{-2}=(\gamma _{ij})_{1\leq i,j\leq n},\,\  \\
& B^{-1}A^{k}d(x_{0},x_{k_{0}})+A^{p}d(x_{0},x_{k_{0}})=\varphi
_{k,p}=(\varphi _{k,p}^{i})_{1\leq i\leq n}.
\end{align*}%
Hence
\begin{equation}
\sum_{i=1}^{n}\varphi _{k,p}^{i}\rightarrow 0\ \ \ \text{as\ \ }%
k,p\rightarrow \infty .  \label{R1}
\end{equation}%
Under these notations, relation \eqref{ineq perov} gives
\begin{equation}
\sum_{j=1}^{n}(\gamma _{ij}-\alpha )d_{j}\left( x_{k},x_{p}\right) \leq
\varphi _{k,p}^{i}\,,\,\,i=1,2,\ldots ,n.
\label{inegalitati sume inmultire matrice}
\end{equation}%
Summing in \eqref{inegalitati sume inmultire matrice} over all $i\in \left\{
1,2,\ldots ,n\right\} $, we obtain
\begin{equation}
\sum_{i,j=1}^{n}(\gamma _{ij}-\alpha )d_{j}(x_{k},x_{p})\leq
\sum_{i=1}^{n}\varphi _{k,p}^{i}.  \label{suma elemnte matrice}
\end{equation}%
Since $B^{-2}$ is invertible and positive, the sum of its elements in each
column must be positive, i.e.,
\begin{equation*}
\sum_{i=1}^{n}\gamma _{ij}>0,\text{ \ \ }j=1,2,\ldots ,n.
\end{equation*}%
If we denote
\begin{equation*}
\gamma =\min \left\{ \sum_{i=1}^{n}\gamma _{ij}\,:\ \,j=1,2,\ldots
,n\right\} ,
\end{equation*}%
relation \eqref{suma elemnte matrice} implies that
\begin{eqnarray*}
\sum_{i=1}^{n}\varphi _{k,p}^{i} &\geq &\sum_{i,j=1}^{n}\gamma
_{ij}d_{j}(x_{k},x_{p})-n\alpha \sum_{j=1}^{n}d_{j}(x_{k},x_{p}) \\
&=&\sum_{j=1}^{n}\left( \sum_{i=1}^{n}\gamma _{ij}\right)
d_{j}(x_{k},x_{p})-n\alpha \sum_{j=1}^{n}d_{j}(x_{k},x_{p}) \\
&\geq &(\gamma -n\alpha )\sum_{j=1}^{n}d_{j}(x_{k},x_{p}).
\end{eqnarray*}%
Choosing $\alpha <\gamma /n$, one has
\begin{equation}
\sum_{j=1}^{n}d_{j}(x_{k},x_{p})\leq \frac{1}{\gamma -n\alpha }%
\sum_{i=1}^{n}\varphi _{k,p}^{i}.  \label{ineq finala}
\end{equation}%
In \eqref{ineq finala}, we observe that the factor $\frac{1}{\gamma -n\alpha
}$ depends only on $n$ and $B$, whence (\ref{R1}) yields
\begin{equation*}
\sum_{j=1}^{n}d_{j}(x_{k},x_{p})\rightarrow 0\ \ \ \text{as\ \ }%
k,p\rightarrow \infty ,
\end{equation*}%
so the sequence $\left( x_{k}\right) $ is Cauchy.

Case (b): $B$ is positive. One has
\begin{align*}
d(x_{k},x_{p})& \leq Bd(x_{k},x_{k+k_{0}})+Bd(x_{p},x_{k+k_{0}}) \\
& \leq
BA^{k}d(x_{0},x_{k_{0}})+B^{2}d(x_{p},x_{p+k_{0}})+B^{2}d(x_{p+k_{0}},x_{k+k_{0}})
\\
& \leq
BA^{k}d(x_{0},x_{k_{0}})+B^{2}A^{p}d(x_{0},x_{k_{0}})+B^{2}A^{k_{0}}d(x_{k},x_{p})
\\
& \leq BA^{k}d(x_{0},x_{k_{0}})+B^{2}A^{p}d(x_{0},x_{k_{0}})+B^{2}\Lambda
d(x_{k},x_{p}),
\end{align*}%
which gives
\begin{equation}
(I-B^{2}\Lambda )d(x_{k},x_{p})\leq
BA^{k}d(x_{0},x_{k_{0}})+B^{2}A^{p}d(x_{0},x_{k_{0}}).  \label{rpf}
\end{equation}%
Note that since $\Lambda ^{k}=(n\alpha )^{k-1}\Lambda $, if $\alpha $ is
chosen to be smaller than one divided by the greatest element of $B^{2}$
multiplied with $n$, the matrix $B^{2}\Lambda $ is convergent to zero.
Consequently, $I-B^{2}\Lambda $ is invertible and $\left( I-B^{2}\Lambda
\right) ^{-1}\in \mathcal{M}_{n\times n}\left(
%TCIMACRO{\U{211d} }%
%BeginExpansion
\mathbb{R}
%EndExpansion
_{+}\right) .$ Hence, (\ref{rpf}) is equivalent to
\begin{equation}
d(x_{k},x_{p})\leq \left( I-B^{2}\Lambda \right) ^{-1}\left(
BA^{k}d(x_{0},x_{k_{0}})+B^{2}A^{p}d(x_{0},x_{k_{0}})\right) .
\label{ineq finala 1}
\end{equation}%
As the right-hand side of \eqref{ineq finala 1} converges to zero when $%
k,p\rightarrow \infty ,$ we conclude that $\left( x_{k}\right) $ is Cauchy.

Therefore, in both cases, the sequence $\left( x_{k}\right) $ is Cauchy and
since $X$ is complete, it has a limit $x^{\ast },$ that is, $d\left(
x_{k},x^{\ast }\right) \rightarrow 0$ as $k\rightarrow \infty .$ Then, from%
\begin{equation*}
d\left( N\left( x_{k}\right) ,N\left( x^{\ast }\right) \right) \leq Ad\left(
x_{k},x^{\ast }\right) ,
\end{equation*}%
it follows that $N\left( x_{k}\right) \rightarrow N\left( x^{\ast }\right) $
as $k\rightarrow \infty ,$ while from $x_{k+1}=N\left( x_{k}\right) ,$
passing to the limit, one obtains $x^{\ast }=N\left( x^{\ast }\right) .$
Hence $N$ has a fixed point. To prove uniqueness, suppose that there exists
another fixed point $x^{\ast \ast }$. Then, from
\begin{equation*}
d\left( x^{\ast },x^{\ast \ast }\right) =d\left( N\left( x^{\ast }\right)
,N\left( x^{\ast \ast }\right) \right) \leq Ad\left( x^{\ast },x^{\ast \ast
}\right) ,
\end{equation*}%
recursively, we obtain that%
\begin{equation*}
d\left( x^{\ast },x^{\ast \ast }\right) \leq A^{k}d\left( x^{\ast },x^{\ast
\ast }\right) ,
\end{equation*}%
for all $k\geq 1.$ Since $A^{k}\rightarrow 0_{n}$ as $k\rightarrow \infty ,$
we deduce that $d\left( x^{\ast },x^{\ast \ast }\right) =0,$ i.e., $x^{\ast
\ast }=x^{\ast }.$
\end{proof}

If we are not interested in the uniqueness of the fixed point for $N$, the
condition \eqref{conditie contractie matrice} can be relaxed and replaced by
a weaker assumption on the graph of $N$.

\begin{Theorem}
\label{Graph} Let $(X,d)$ be a complete vector $B$-metric space, where $B$
is either positive or inverse-positive, and let $N\colon X\rightarrow X$ be
an operator. Assume there exists a convergent to zero matrix $A\in \mathcal{M%
}_{n\times n}(\mathbb{R}_{+})$ such that
\begin{equation}
d\left( N(x),N^{2}(x)\right) \leq Ad(x,N(x)),\ \,\,\text{for all }x\in X.
\label{cg}
\end{equation}%
Then, $N$ has at least one fixed point.
\end{Theorem}

\begin{proof}
Following the proof of Theorem \ref{Perov}, from any initial point $x_{0},$
the sequence $x_{k}=N^{k}\left( x_{0}\right) $ is convergent to a fixed
point $x^{\ast }$ of $N,$ which clearly depends on the starting point $%
x_{0}, $ but condition (\ref{cg}) is insufficient to guarantee the
uniqueness.
\end{proof}

The next result is a version for vector $B$-metric spaces of Maia's fixed
point theorem. The contraction condition on the operator is considered with
respect to a vector $B_{1}$-metric $d_{1},$ not necessarily complete, while
the convergence of the sequence of successive approximations is guaranteed
in a complete vector $B_{2}$-metric $d_{2}$ in a subordinate relationship to
$d_{1}.$

\begin{Theorem}
\label{maia} Let $X$ be a set equipped with two $%
%TCIMACRO{\U{211d} }%
%BeginExpansion
\mathbb{R}
%EndExpansion
^{n}$-vector metrics, a $B_{1}$-metric $d_{1}$ and a $B_{2}$-metric $d_{2}$,
where $B_{2}$ is either positive or inverse-positive, and let $N\colon
X\rightarrow X$ be an operator. Assume that the following conditions hold:
\begin{enumerate}
\item[(i)] $(X,d_{1})$ is a complete vector $B_{1}$-metric space;

\item[(ii)] $d_{1}(x,y)\leq Cd_{2}(x,y)$ for all $x,y\in X$ and some matrix $%
C\in \mathcal{M}_{n\times n}\left(
%TCIMACRO{\U{211d} }%
%BeginExpansion
\mathbb{R}
%EndExpansion
\right) ;$
\item[(iii)] There exists a matrix $A$ convergent to zero such that
\begin{equation}
d_{2}\left( N(x),N(y)\right) \leq Ad_{2}(x,y),\ \text{ for all }x,y\in X;
\label{rpma}
\end{equation}
\item[(iv)] The operator $N$ is continuous in $\left(X,d_1\right)$.
\end{enumerate}
Then, the operator $N$ has a unique fixed point.
\end{Theorem}

\begin{proof}
Let $x_{0}\in X$ be fixed, and consider the iterative sequence $%
x_{k+1}=N(x_{k})$ for $k\geq 0.$ For any $k,k_{0},p\geq 0$, applying the
triangle inequality twice and using condition (iii), we derive either
\begin{equation*}
(B_{2}^{-2}-A^{k_{0}})d_{2}(x_{k},x_{p})\leq
B_{2}^{-1}A^{k}d_{2}(x_{0},x_{k_{0}})+A^{p}d_{2}(x_{0},x_{k_{0}}),
\end{equation*}%
in case that $B_{2}$ is inverse-positive, or%
\begin{equation*}
(I-B_{2}^{2}A^{k_{0}})d_{2}(x_{k},x_{p})\leq
B_{2}A^{k}d_{2}(x_{0},x_{k_{0}})+B_{2}^{2}A^{p}d_{2}(x_{0},x_{k_{0}}),
\end{equation*}%
if $B_{2}$ is positive. Arguing similarly to the proof of Theorem \ref{Perov}%
, we deduce that $(x_{k})$ is a Cauchy sequence in $(X,d_{2})$. From (ii),
it follows immediately that $(x_{k})$ is also a Cauchy sequence in $(X,d_{1})
$, hence $(x_{k})$ is convergent with respect the metric $d_{1}$ to some $%
x^{\ast },$ that is,
\begin{equation*}
d_{1}(N(x_{k}),x^{\ast })=d_{1}\left( x_{k+1},x^{\ast }\right) \rightarrow
0,\ \ \ \text{as }k\rightarrow \infty ,
\end{equation*}%
while the continuity of $N$ yields $d_{1}(N(x^{\ast }),x^{\ast })=0$, i.e., $%
N(x^{\ast })=x^{\ast }$. To establish uniqueness, suppose that $x^{\ast \ast
}$ is another fixed point of $N,$ i.e., $N(x^{\ast \ast })=x^{\ast \ast }$.
Then, by (\ref{rpma}), one has
\begin{equation*}
(I-A)d_{2}(x^{\ast },x^{\ast \ast })\leq 0.
\end{equation*}%
Since $A$ is convergent to zero, we necessarily have $d_{2}(x^{\ast
},x^{\ast \ast })=0$, i.e., $x^{\ast }=x^{\ast \ast }$.
\end{proof}

\subsection{Error estimates}

The classical Banach and Perov fixed point theorems are accompanied by some
error estimates in terms of the contraction constant and matrix,
respectively. These estimates allow us to obtain stopping criteria for the
iterative approximation process. It is the aim of this subsection to obtain
such stopping criteria when working in vector $B$-metric spaces.

\begin{Theorem}
Assume that all the conditions of Theorem \ref{Perov} hold and let $\left(
x_{k}\right) $ be a sequence of successive approximations of the fixed point
$x^{\ast }.$

\begin{description}
\item[(1$^{0}$)] If $B$ is inverse-positive, then
\begin{equation}
\left( B^{-1}-A\right) d\left( x_{k},x^{\ast }\right) \leq A^{k}d\left(
x_{0},x_{1}\right) \ \ \ \ \left( k\geq 0\right) .  \label{rp1}
\end{equation}%
If in addition the matrix $B^{-1}-A$ is inverse-positive, then%
\begin{equation}
d\left( x_{k},x^{\ast }\right) \leq \left( B^{-1}-A\right) ^{-1}A^{k}d\left(
x_{0},x_{1}\right) \ \ \ \left( k\geq 0\right) .  \label{rp2}
\end{equation}

\item[(2$^{0}$)] If $B$ is positive, then%
\begin{equation}
\left( I-BA\right) d\left( x_{k},x^{\ast }\right) \leq BA^{k}d\left(
x_{0},x_{1}\right) \ \ \ \ \left( k\geq 0\right) .  \label{rp3}
\end{equation}%
If in addition $\ I-BA$ is inverse-positive, then%
\begin{equation}
d\left( x_{k},x^{\ast }\right) \leq \left( I-BA\right) ^{-1}BA^{k}d\left(
x_{0},x_{1}\right) \ \ \ \left( k\geq 0\right) .  \label{rp4}
\end{equation}
\end{description}
\end{Theorem}

\begin{proof}
(1$^{0})$: We have%
\begin{eqnarray*}
B^{-1}d\left( x_{k},x^{\ast }\right) &\leq &d\left( x_{k},x_{k+1}\right)
+d\left( x_{k+1},x^{\ast }\right) \\
&\leq &A^{k}d\left( x_{0},x_{1}\right) +Ad\left( x_{k},x^{\ast }\right) ,
\end{eqnarray*}%
whence we deduce (\ref{rp1}). The second part is obvious.

(2$^{0}$): We have%
\begin{eqnarray*}
d\left( x_{k},x^{\ast }\right) &\leq &Bd\left( x_{k},x_{k+1}\right)
+Bd\left( x_{k+1},x^{\ast }\right) \\
&\leq &BA^{k}d\left( x_{0},x_{1}\right) +BAd\left( x_{k},x^{\ast }\right) ,
\end{eqnarray*}%
that is (\ref{rp3}). The additional conclusion is obvious.
\end{proof}

\begin{Remark}
Clearly, since $A^{k}$ tends to the zero matrix as $k\rightarrow \infty ,$
formulas (\ref{rp2}) and (\ref{rp4}) provide stopping criteria for the
iterative fixed point approximation algorithm starting from $x_{0},$ when an
admissible error is given. It should be emphasized that these estimates are
in terms of matrices $A$ and $B$. In contrast, if we make the transition to
(scalar) $b$-metric spaces, as discussed in Section 2, the resulting
estimates will depend on the chosen norm in $\mathbb{R}^{n}$ and may vary
across different norms. So, from this point of view, the vector approach not
only unifies the results that can be obtained with the scalar method, but
also provides the best estimates.
\end{Remark}

\subsection{Stability results}

We now present two stability properties of the Perov contraction mappings in
vector $B$-metric spaces.

The first property is in the sense of Reich and Zaslavski and generalizes
the one obtained in \cite{pp} for $b$-metric spaces.

\begin{Theorem}
\label{stabilitate1} Let $(X,d)$ be a complete vector $B$-metric space, and
let $N\colon X\rightarrow X$ be an operator such that
\eqref{conditie
contractie matrice} holds with a matrix $A$ convergent to zero. In addition
assume that either

\begin{description}
\item[(a)] $B$ and $B^{-1}-A$ are inverse-positive;
\end{description}

or

\begin{description}
\item[(b)] $B$ is positive and $I-BA$ is inverse-positive.
\end{description}

Then, $N$ is stable in the sense of Reich and Zaslavski, i.e., $N$ has a
unique fixed point $x^{\ast }$, and for every sequence $(x_{k})\subset X$
satisfying
\begin{equation}
d(x_{k},N(x_{k}))\rightarrow 0\text{ }\ \ \text{as\ \ }k\rightarrow \infty ,
\label{R5}
\end{equation}%
one has
\begin{equation*}
x_{k}\rightarrow x^{\ast }\ \ \ \text{ as \ }k\rightarrow \infty .
\end{equation*}
\end{Theorem}

\begin{proof}
According to Theorem \ref{Perov} the operator $N$ has a unique fixed point $%
x^{\ast }.$ In addition, for any sequence $\left( x_{k}\right) $ satisfying (%
\ref{R5}), in case (a), we have
\begin{align*}
B^{-1}d(x_{k},x^{\ast })& \leq d(x_{k},N(x_{k}))+d(N(x_{k}),x^{\ast }) \\
& =d(x_{k},N(x_{k}))+d(N(x_{k}),N(x^{\ast })) \\
& \leq d(x_{k},N(x_{k}))+Ad(x_{k},x^{\ast }),
\end{align*}%
that is,
\begin{equation*}
d(x_{k},x^{\ast })\leq (B^{-1}-A)^{-1}d(x_{k},N(x_{k})),
\end{equation*}%
while in case (b),%
\begin{equation*}
d(x_{k},x^{\ast })\leq \left( I-BA\right) ^{-1}Bd(x_{k},N(x_{k})).
\end{equation*}%
These estimates immediately yield the conclusion.
\end{proof}

The second stability result is in the sense of Ostrowski and extends to
vector $B$-metric spaces a similar property established in \cite{pp} for $b$%
-metric spaces.

\begin{Theorem}
\label{stabilitate2} Let $(X,d)$ be a complete vector $B$-metric space, and
let $N\colon X\rightarrow X$ be an operator. Assume $N$ satisfies
\eqref{conditie
contractie matrice} with a matrix $A$ convergent to zero. In addition,
assume that either
\begin{description}
\item[(a)] $B$ and $I-\Tilde{b}A$ are inverse-positive, where $\Tilde{b}%
=\max \{b_{ii}\,:\,i=1,2,\ldots ,n\}$;
\end{description}
or
\begin{description}
\item[(b)] $B$ is positive and $I-BA$ is inverse-positive.
\end{description}
Then, $N$ has the Ostrowski property, i.e., $N$ has a unique fixed point $%
x^{\ast }$, and for every sequence $(x_{k})\subset X$ satisfying
\begin{equation*}
d(x_{k+1},N(x_{k}))\rightarrow 0\text{ \ as \ }k\rightarrow \infty ,
\end{equation*}%
one has
\begin{equation*}
x_{k}\rightarrow x^{\ast }\ \ \ \text{ as \ }k\rightarrow \infty .
\end{equation*}
\end{Theorem}

\begin{proof}
As previously established, the operator $N$ has a unique fixed point $%
x^{\ast }$. In case (a), we have
\begin{align*}
d(x_{k+1},x^{\ast })& \leq \Tilde{b}\,d(x_{k+1},N(x_{k}))+\Tilde{b}%
\,d(N(x_{k}),N(x^{\ast })) \\
& \leq \Tilde{b}\,d(x_{k+1},N(x_{k}))+\Tilde{b}A\,d(x_{k},x^{\ast }) \\
& \leq \ \ldots  \\
& \leq \Tilde{b}\sum_{p=0}^{k}(\Tilde{b}A)^{p}d(x_{k+1-p},N(x_{k-p}))+(%
\Tilde{b}A)^{k+1}d(x_{0},x^{\ast }),
\end{align*}%
while in case (b), similar estimation gives
\begin{equation*}
d(x_{k+1},x^{\ast })\leq
\sum_{p=0}^{k}(BA)^{p}Bd(x_{k+1-p},N(x_{k-p}))+(BA)^{k}Bd(x_{0},x^{\ast }).
\end{equation*}%
% In the first case, since $(I-\Tilde{b}A)$ is inverse positive and $\Tilde{b}A$ is positive, the serai $\sum_{p=0}^{k}(\Tilde{b}A)^{p}$ is convergent and $(\Tilde{b}A)^k$ is convergent to zero.
% Same arguments in the second case shows that the series
% $\sum_{p=0}^{k}(BA)^{p}$ are convergent and
% $(AB)^k$ converge to the zero matrix as $k\to \infty.$

Since $I - \tilde{b}A $ is inverse-positive and $\tilde{b}A $ is positive in
the first case, and $I - BA $ is inverse-positive and $BA $ is positive in
the second case, the series $\sum_{p=0}^{k} (\tilde{b}A)^p $ and $%
\sum_{p=0}^{k} (BA)^p $ are convergent. Moreover, $(\tilde{b}A)^k $ and $%
(BA)^k $ converge to the zero matrix as $k \to \infty $.
% In both cases, since $(I-\Tilde{b}A)$ and $I-BA$ are inverse positive as well as $\Tilde{b}A$ and $BA$ positive matrix, respectivelyu , the series $\sum_{p=0}^{k}(\Tilde{b}A)^{p}$ and $\sum_{p=0}^{k}(BA)^{p}$ are convergent to $(I - \tilde{b}A)^{-1}$ and $(I-AB)^{-1}$, respectively. Also, the same argument ensures that both $(\Tilde{b}A)^k$ and $(AB)^k$ converge to the zero matrix as $k\to \infty.$
% In case (a), since \( r(\tilde{b}A) = \tilde{b} r(A) < 1 \), the series \( \sum_{p=0}^{k} (\tilde{b}A)^p \) converges to \( (I - \tilde{b}A)^{-1} \), and \( (\tilde{b}A)^k  \) converges to the zero matrix.  Analogously, in case (b), the same conclusion holds  since \( r(AB) < 1 \).
Therefore, using the Cauchy-Toeplitz lemma (see \cite{ortega}), it follows
that $d(x_{k+1},x^{\ast })\rightarrow 0$ as $k\rightarrow \infty .$
\end{proof}

\subsection{Avramescu type fixed point theorem}

Our next result is a variant of Avramescu's fixed point theorem (see \cite%
{avramescu}) in vector $B$-metric spaces.

\begin{Theorem}[Avramescu theorem in vector $B$-metric spaces]
\label{Avramescu} Let $(X,d)$ be a complete vector $B$-metric space, $D$ a
nonempty closed convex subset of a normed space $Y,$ $N_{1}:X\times
D\rightarrow X$ and $N_{2}:X\times D\rightarrow D$ be two mappings. Assume
that the following conditions are satisfied:

\begin{enumerate}
\item[(i)] $N_{1}\left( x,.\right) $ is continuous for every $x\in X$ and
there is a matrix $A$ convergent to zero such that
\begin{equation*}
d(N_{1}(x,y),N_{1}(\overline{x},y))\leq A\,d(x,\overline{x}),
\end{equation*}%
for all $x,\overline{x}\in X$ and $y\in D;$
\item[(ii)] Either
\end{enumerate}

\begin{description}
\item[(a)] $B$ and $B^{-1}-A$ is inverse-positive;
\end{description}

or

\begin{description}
\item[(b)] $B$ is positive and $I-BA$ is inverse-positive.
\end{description}

\begin{enumerate}
\item[(iii)] $N_{2}$ is continuous and $N_{2}(X\times D)$ is a relatively
compact subset of $Y$ .
\end{enumerate}
Then, there exists $(x^{\ast },y^{\ast })\in X\times D$ such that
\begin{equation*}
N_{1}(x^{\ast },y^{\ast })=x^{\ast },\quad N_{2}(x^{\ast },y^{\ast
})=y^{\ast }.
\end{equation*}
\end{Theorem}

\begin{proof}
For each $y\in D$, Theorem \ref{Perov} applies to the operator $N_{1}\left(
.,y\right) $ and gives a unique $S(y)\in X$ such that
\begin{equation}
N_{1}(S(y),y)=S(y).  \label{R2}
\end{equation}%
We claim that the mapping $S:D\rightarrow X$ is continuous. To prove this,
let $y,\overline{y}\in D$. In case (a), we have
\begin{align*}
B^{-1}d(S(y),S(\overline{y}))& =B^{-1}d\left( N_{1}\left( S(y),y\right)
,N_{1}\left( S(\overline{y}),\overline{y}\right) \right)  \\
& \leq d\left( N_{1}\left( S(y),y\right) ,N_{1}\left( S(\overline{y}%
),y\right) \right) +d\left( N_{1}\left( S(\overline{y}),y\right)
,N_{1}\left( S(\overline{y}),\overline{y}\right) \right)  \\
& \leq Ad\left( S(y),S(\overline{y})\right) +d\left( N_{1}\left( S(\overline{%
y}),y\right) ,N_{1}\left( S(\overline{y}),\overline{y}\right) \right) ,
\end{align*}%
which implies
\begin{equation*}
\left( B^{-1}-A\right) d(S(y),S(\overline{y}))\leq d\left( N_{1}\left( S(%
\overline{y}),y\right) ,N_{1}\left( S(\overline{y}),\overline{y}\right)
\right) ,
\end{equation*}%
while in case (b), one has
\begin{equation*}
\left( I-BA\right) d(S(y),S(\overline{y}))\leq Bd\left( N_{1}\left( S(%
\overline{y}),y\right) ,N_{1}\left( S(\overline{y}),\overline{y}\right)
\right) .
\end{equation*}%
Since $B^{-1}-A$ and $I-BA$ are inverse-positive, respectively, in case (a),
we deduce that
\begin{equation}
d(S(y),S(\overline{y}))\leq \left( B^{-1}-A\right) ^{-1}d\left( N_{1}\left(
S(\overline{y}),y\right) ,N_{1}\left( S(\overline{y}),\overline{y}\right)
\right) ,  \label{avramescu case a}
\end{equation}%
and in case (b),%
\begin{equation}
d(S(y),S(\overline{y}))\leq \left( I-BA\right) ^{-1}Bd\left( N_{1}\left( S(%
\overline{y}),y\right) ,N_{1}\left( S(\overline{y}),\overline{y}\right)
\right) .  \label{avramescu case b}
\end{equation}%
Then, for any convergent sequence $\left( y_{k}\right) \subset D,$ $%
y_{k}\rightarrow y^{\ast }$ as $k\rightarrow \infty ,$ the continuity of $%
N_{1}\left( S\left( y^{\ast }\right) ,.\right) $ together with relations %
\eqref{avramescu case a} and \eqref{avramescu case b} implies that $%
d(S(y_{k}),S(y^{\ast }))\rightarrow 0$ as $k\rightarrow \infty .$
% \begin{equation*}
% d(S(y_{k}),S(y^{\ast }))\leq \left( B^{-1}-A\right) ^{-1}d\left( N_{1}\left(
% S(y^{\ast }),y_{k}\right) ,N_{1}\left( S(y^{\ast }),y^{\ast }\right) \right)
% ,
% \end{equation*}%
% while in case (b),%
% \begin{equation*}
% d(S(y_{k}),S(y^{\ast }))\leq \left( I-BA\right) ^{-1}Bd\left( N_{1}\left(
% S(y^{\ast }),y_{k}\right) ,N_{1}\left( S(y^{\ast }),y^{\ast }\right) \right)
% .
% \end{equation*}%
% Since $N_{1}\left( S\left( y^{\ast }\right) ,.\right) $ is continuous, we
% deduce that in both cases,
Thus, $S$ is continuous, and since $N_{2}$ is continuous, the composed
mapping
\begin{equation*}
N_{2}(S(.),.):D\rightarrow D
\end{equation*}%
is continuous too. Since its range is relatively compact by condition (iii),
Schauder's fixed point theorem applies and guarantees the existence of a
point $y^{\ast }\in D$ such that
\begin{equation}
N_{2}\left( S(y^{\ast }),y^{\ast }\right) =y^{\ast }.  \label{R3}
\end{equation}%
Finally, denoting $x^{\ast }:=S(y^{\ast }),$ from (\ref{R2}) and (\ref{R3})
we have the conclusion.
\end{proof}

\begin{Remark}
Without the invariance condition $N_{2}\left( X\times D\right) \subset D,$ a
similar result holds if $D$ is a closed ball $B_{R}$ centered at the origin
and of radius $R$ in the space $\left( Y,\left\Vert .\right\Vert \right) ,$
provided that Schaefer's fixed point theorem is used instead of Schauder's
theorem. In this case, in addition to conditions (i) and (ii), we need the
Leray-Schauder condition
\begin{equation*}
y\neq \lambda N_{2}\left( x,y\right) ,
\end{equation*}%
\ for all $x\in X,$ $y\in Y$ with $\left\Vert y\right\Vert =R,$ and $\lambda
\in \left( 0,1\right) .$
\end{Remark}

In particular, for scalar $b$-metric spaces, conditions (a) and (b) from
hypothesis (ii) of Theorem \ref{Avramescu} are the same and reduce to the
unique requirement that the product of $b$ and the Lipschitz constant $a$ of
$N$ is less than one. More exactly, Theorem \ref{Avramescu} reads as follows.

\begin{Theorem}[Avramescu theorem in $b$-metric spaces]
Let $(X,\rho )$ be a complete $b$-metric space $\left( b\geq 1\right) $, $D$
a nonempty closed convex subset of a normed space $Y,$ $N_{1}:X\times
D\rightarrow X$ and $N_{2}:X\times D\rightarrow D$ be two mappings. Assume
that the following conditions are satisfied:

\begin{enumerate}
\item[(i)] $N_{1}\left( x,.\right) $ is continuous for every $x\in X$ and
there is a constant $a\geq 0$ such that
\begin{equation*}
\rho (N_{1}(x,y),N_{2}(\overline{x},y))\leq a\rho (x,\overline{x}),
\end{equation*}%
for all $x,\overline{x}\in X$ and $y\in D;$

\item[(ii)] $ab<1;$

\item[(iii)] $N_{2}$ is continuous and $N_{2}(X\times D)$ is a relatively
compact subset of $Y$ .
\end{enumerate}

Then, there exists $(x^{\ast },y^{\ast })\in X\times D$ such that $\
N_{1}(x^{\ast },y^{\ast })=x^{\ast }$ and $\ N_{2}(x^{\ast },y^{\ast
})=y^{\ast }.$
\end{Theorem}

\section{Ekeland's principle and Caristi's fixed point theorem in vector $B$%
-metric spaces}

\subsection{Classical results}

We first recall for comparison the classical results in metric spaces (see,
\cite{ekeland,mawhin,f,meghea}).

\begin{Theorem}[Weak Ekeland variational principle]
\label{te}Let $\ \left( X,\rho \right) $ be a complete metric space and let $%
\ f:X\rightarrow \mathbf{%
%TCIMACRO{\U{211d} }%
%BeginExpansion
\mathbb{R}
%EndExpansion
}$ be a lower semicontinuous function bounded from below. Then, for given $\
\varepsilon >0$ and $\ x_{0}\in X,$ there exists a point $\ x^{\ast }\in X$
such that%
\begin{equation*}
f\left( x^{\ast }\right) \leq f\left( x_{0}\right) -\varepsilon \rho \left(
x^{\ast },x_{0}\right)
\end{equation*}%
and%
\begin{equation*}
f\left( x^{\ast }\right) <f\left( x\right) +\varepsilon \rho \left( x^{\ast
},x\right) \ \ \ \text{for all \ }x\in X,\ x\neq x^{\ast }.
\end{equation*}
\end{Theorem}

\begin{Theorem}[Strong Ekeland variational principle]
Let $(X,\rho )$ be a complete metric space, and let $f:X\rightarrow \mathbb{R%
}$ be a lower semicontinuous function that is bounded from below. For given $%
\varepsilon >0$, $\delta >0$, and $x_{0}\in X$ satisfying
\begin{equation*}
f(x_{0})\leq \inf_{x\in X}f(x)+\varepsilon ,
\end{equation*}%
there exists a point $x^{\ast }\in X$ such that the following hold:
\begin{align*}
& f(x^{\ast })\leq f(x_{0}), \\
& \rho (x^{\ast },x_{0})\leq \delta , \\
& f(x^{\ast })<f(x)+\frac{\varepsilon }{\delta }\rho (x^{\ast },x)\ \,\,\
\text{for all }x\in X,\ x\neq x^{\ast }.
\end{align*}
\end{Theorem}

Below, we have a version of Ekeland's variational principle for scalar $b$%
-metric spaces (see, \cite{bmv}).

\begin{Theorem}[\protect\cite{bmv}]
\label{bmv} Let $(X,\rho )$ be a complete $b$-metric space with $b>1$, where
the $b$-metric $\rho $ is continuous. Let $f:X\rightarrow \mathbb{R}$ be a
lower semicontinuous function bounded from below. For a given $\varepsilon >0
$ and $x_{0}\in X$ satisfying
\begin{equation*}
f(x_{0})\leq \inf_{x\in X}f(x)+\varepsilon ,
\end{equation*}%
there exists a sequence $(x_{k})\subset X$ and a point $x^{\ast }\in X$ such
that:
\begin{align*}
& x_{k}\rightarrow x^{\ast }\quad \text{as }k\rightarrow \infty , \\
& \rho (x^{\ast },x_{k})\leq \frac{\varepsilon }{2^{k}},\quad k\in \mathbb{N}%
, \\
& f(x^{\ast })\leq f(x_{0})-\sum_{k=0}^{\infty }\frac{1}{b^{k}}\rho (x^{\ast
},x_{k}), \\
& f(x^{\ast })+\sum_{k=0}^{\infty }\frac{1}{b^{k}}\rho (x^{\ast
},x_{k})<f(x)+\sum_{k=0}^{\infty }\frac{1}{b^{k}}\rho (x,x_{k}),\quad \text{%
for }x\neq x^{\ast }.
\end{align*}
\end{Theorem}

The proof of Theorem \ref{bmv} in \cite{bmv} is based on the version for
scalar $b$-metric spaces of Cantor's intersection lemma.

\begin{Lemma}[\protect\cite{bmv}]
\label{Cantor} Let $\left( X,\rho \right) $ be a complete $b$-metric space.
For every descending sequence $\left( F_{k}\right) _{k\geq 1}$ of nonempty
closed subsets of $X$ with \textrm{diam}$_{\rho }\left( F_{k}\right)
\rightarrow 0$ as $k\rightarrow \infty ,$ the intersection $%
\bigcap\limits_{k=1}^{\infty }F_{k}$ contains one and only one element.
\end{Lemma}

Let us first note that a version of Cantor's intersection lemma remains true
in complete vector $B$-metric spaces.
% Indeed, as mentioned in the Preliminaries, if the sequence $F_k$ has the property that  \begin{equation*}
%     d(x,y)\leq \varepsilon_k e,
% \end{equation*}
% for some sequence $\varepsilon_k$ convergent to zero, then the intersection has a unique element, conclusion obtained applying Lemma \ref{Cantor} for the $b$-metric space $(X,\phi_d)$.
% Thus, we have

\begin{Lemma}
Let $\left( X,d\right) $ be a complete vector $B$-metric space, and let $%
\left( F_{k}\right) _{k\geq 1}$ be a descending sequence of nonempty closed
subsets of $X$. Assume that for every $\varepsilon >0$, there exists $%
k_{0}\geq 1$ such that
\begin{equation}
d(x,y)\leq \varepsilon e\ \ \text{ for all }x,y\in F_{k}\text{ and }k\geq
k_{0},  \label{cond diametru vectorial}
\end{equation}%
where $e=(1,1,\ldots ,1)$. Then, the intersection $\bigcap\limits_{k=1}^{%
\infty }F_{k}$ contains exactly one element.
\end{Lemma}

\begin{proof}
As stated in the Preliminaries, condition \eqref{cond diametru vectorial}
implies that the diameter of $F_{k}$ with respect to the scalar $b$-metric $%
\rho _{1}$ tends to zero. Since $(X,d)$ is complete, it follows that $%
(X,\rho _{1})$ is also complete. From Cantor's lemma in scalar $b$-metric
spaces (Lemma \ref{Cantor}), we conclude that the intersection $%
\bigcap\limits_{k=1}^{\infty }F_{k}$ has exactly one element.
\end{proof}

\subsection{Ekeland variational principle in vector $B$-metric spaces}

First we state and prove a version of the weak form of Ekeland's variational
principle in vector $B$-metric spaces.

\begin{Theorem}[Weak Ekeland variational principle in vector $B$-metric
spaces]
\label{ep}Let $\left( X,d\right) $ be a complete vector $B$-metric space
such that the $B$-metric $d$ is continuous, and let $f:X\rightarrow
%TCIMACRO{\U{211d} }%
%BeginExpansion
\mathbb{R}
%EndExpansion
^{n}$ be a lower semicontinuous function bounded from below. Assume that $f$
satisfies the following condition:

\begin{description}
\item[(H)] For every nonempty closed subset $F \subset X $ and every $%
\varepsilon > 0 $, there exists a point $x_{\varepsilon, F} \in F $ such
that
\begin{equation}  \label{h1}
f(x_{\varepsilon, F}) \leq f(x) + \varepsilon e, \quad \text{for all } x \in
F,
\end{equation}
where $e = (1, 1, \ldots, 1) \in \mathbb{R}^n $.
\end{description}

Then, for a given $x_0 \in X $, there exists a sequence $\{x_k\} \subset X $
and a point $x^\ast \in X $ such that $x_k \to x^\ast $ as $k \to \infty $,
\begin{equation}
f(x^\ast) \leq f(x_0) - d(x^\ast, x_0),  \label{c1}
\end{equation}
and
\begin{equation}
f(x^\ast) + d(x^\ast, x_k) \geq f(x) + d(x, x_k) \quad \text{for all } k
\geq 0 \quad \text{implies } x = x^\ast.  \label{c2}
\end{equation}
Moreover,
\begin{equation}
f(x^\ast) \geq f(x) + B d(x^\ast, x) + (B - I) d(x^\ast, x_k) \quad \text{%
for all } k \geq 0 \quad \text{implies } x = x^\ast.  \label{c3}
\end{equation}
\end{Theorem}

\begin{proof}
Let us fix a sequence $\left( \varepsilon _{k}\right) $ of positive numbers
satisfying $\varepsilon _{k}\rightarrow 0$ as $k\rightarrow \infty .$ We now
proceed to construct the sequence $\left( x_{k}\right) .$ Let%
\begin{equation*}
F\left( x_{0}\right) :=\left\{ x\in X:\ f\left( x\right) +d\left(
x,x_{0}\right) \leq f\left( x_{0}\right) \right\} .
\end{equation*}%
Clearly, $x_{0}\in F\left( x_{0}\right) $ and $F\left( x_{0}\right) $ is
closed because $d$ is continuous and $f$ is lower semicontinuous. Then, by
assumption (\ref{h1}), there exists a point $x_{1}\in F\left( x_{0}\right) $
with%
\begin{equation*}
f\left( x_{1}\right) \leq f\left( x\right) +\varepsilon _{1}e\ \ \ \text{for
all }x\in F\left( x_{0}\right) .
\end{equation*}%
Define%
\begin{equation*}
F\left( x_{1}\right) :=\left\{ x\in F\left( x_{0}\right) :\ f\left( x\right)
+d\left( x,x_{1}\right) \leq f\left( x_{1}\right) \right\} ,
\end{equation*}%
and recursively, having $x_{k}\in F\left( x_{k-1}\right) $ with%
\begin{equation*}
f\left( x_{k}\right) \leq f\left( x\right) +\varepsilon _{k}e\ \ \ \text{for
all }x\in F\left( x_{k-1}\right) ,
\end{equation*}%
we define%
\begin{equation*}
F\left( x_{k}\right) :=\left\{ x\in F\left( x_{k-1}\right) :\ f\left(
x\right) +d\left( x,x_{k}\right) \leq f\left( x_{k}\right) \right\} .
\end{equation*}%
The sets $F\left( x_{k}\right) $ are nonempty and closed, and by their
definition form a descending sequence. To apply Cantor's intersection lemma,
we verify that their diameters tend to zero as $k\rightarrow \infty .$
Indeed, for any $y\in F\left( x_{k}\right) \subset F\left( x_{k-1}\right) ,$
one has
\begin{equation*}
f\left( y\right) +d\left( y,x_{k}\right) \leq f\left( x_{k}\right) .
\end{equation*}%
Also, from the definition of $x_{k},$
\begin{equation*}
f\left( x_{k}\right) \leq f\left( y\right) +\varepsilon _{k}e.
\end{equation*}%
Consequently, using the definition of $F(x_{k})$, we deduce
\begin{equation*}
d\left( y,x_{k}\right) \leq f\left( x_{k}\right) -f\left( y\right) \leq
\varepsilon _{k}e,
\end{equation*}%
whence, for every $y_{1},y_{2}\in F\left( x_{k}\right) ,$ we have%
\begin{equation*}
d\left( y_{1},y_{2}\right) \leq B\left( d\left( y_{1},x_{k}\right) +d\left(
y_{2},x_{k}\right) \right) .
\end{equation*}%
As a result, diam$_{d}\left( F\left( x_{k}\right) \right) \rightarrow 0$ as $%
k\rightarrow \infty .$ Thus, by Cantor's lemma,
\begin{equation*}
\bigcap\limits_{k=0}^{\infty }F\left( x_{k}\right) =\left\{ x^{\ast
}\right\} .
\end{equation*}%
From $x^{\ast }\in F\left( x_{0}\right) ,$ one has (\ref{c1}).

Next, we prove (\ref{c2}). To this end, we show the equivalent statement: if
$x\neq x^{\ast },$ then there exists $k=k\left( x\right) \geq 0$ such that
\begin{equation*}
f\left( x^{\ast }\right) +d\left( x^{\ast },x_{k}\right) \ngeqslant f\left(
x\right) +d\left( x,x_{k}\right) ,
\end{equation*}%
that is
\begin{equation*}
f_{i}\left( x^{\ast }\right) +d_{i}\left( x^{\ast },x_{k}\right)
<f_{i}\left( x\right) +d_{i}\left( x,x_{k}\right)
\end{equation*}%
for at least one index $i\in \left\{ 1,2,...,n\right\} .$

Let $x\in X,\ x\neq x^{\ast }$ be arbitrary. Then $x\notin
\bigcap\limits_{k=0}^{\infty }F\left( x_{k}\right) .$ We distinguish two
cases:
\begin{equation*}
\text{(a)\ \ }x\notin F\left( x_{0}\right) ;\ \ \text{(b)\ \ }x\in F\left(
x_{k-1}\right) \ \text{and }x\notin F\left( x_{k}\right) \ \text{for some }%
k=k\left( x\right) \geq 1.
\end{equation*}%
In case (a), we have $f\left( x\right) +d\left( x,x_{0}\right) \nleqslant
f\left( x_{0}\right) .$ In case (b), we have $f\left( x\right) +d\left(
x,x_{k}\right) \nleqslant f\left( x_{k}\right) .$ Thus, in both cases, there
exists $k=k\left( x\right) \geq 0$ such that $f\left( x\right) +d\left(
x,x_{k}\right) \nleqslant f\left( x_{k}\right) .$ This implies that there is
some $i\in \left\{ 1,2,...,n\right\} $ with
\begin{equation*}
f_{i}\left( x\right) +d_{i}\left( x,x_{k}\right) >f_{i}\left( x_{k}\right) .
\end{equation*}%
On the other hand, since $x^{\ast }\in F\left( x_{k}\right) ,$ one has $%
f\left( x^{\ast }\right) +d\left( x^{\ast },x_{k}\right) \leq f\left(
x_{k}\right) .$ In particular, for the index $i$ identified above, it holds
that
\begin{equation*}
f_{i}(x_{k})\geq f_{i}(x^{\ast })+d_{i}(x^{\ast },x_{k}).
\end{equation*}%
Then, from these two ineqialities we obtain
\begin{equation}
f_{i}\left( x^{\ast }\right) +d_{i}\left( x^{\ast },x_{k}\right)
<f_{i}\left( x\right) +d_{i}\left( x,x_{k}\right) ,  \label{idc}
\end{equation}%
which equivalently proves (\ref{c2}).

In order to establish (\ref{c3}), we apply the triangle inequality for $d$
on the right hand side of (\ref{idc}), which gives,
\begin{equation*}
f_{i}\left( x^{\ast }\right) +d_i\left( x^{\ast },x_{k}\right) <f_{i}\left(
x\right) +d_i\left( x,x_{k}\right) \leq f_{i}\left( x\right) +\left(
Bd\left( x^{\ast },x_{k}\right) \right) _{i}+\left( Bd\left( x^{\ast
},x\right) \right) _{i}.
\end{equation*}%
Hence%
\begin{equation*}
f_{i}\left( x^{\ast }\right) +d_i\left( x^{\ast },x_{k}\right) <f_{i}\left(
x\right) +\left( Bd\left( x^{\ast },x_{k}\right) \right) _{i}+\left(
Bd\left( x^{\ast },x\right) \right) _{i},
\end{equation*}%
that is,
\begin{equation*}
f\left( x^{\ast }\right) \ngeqslant f\left( x\right) +Bd\left( x^{\ast
},x\right) +\left( B-I\right) d\left( x^{\ast },x_{k}\right) .
\end{equation*}%
Thus, (\ref{c3}) holds.
\end{proof}

A version of the strong form of Ekeland's variational principle in vector $B$%
-metric spaces is the following one.

\begin{Theorem}[Strong Ekeland variational principle in vector $B$-metric
spaces]
\label{sep}Let $\left( X,d\right) $ be a complete $B$-metric space such that
the $B$-metric $d$ is continuous, and let $f:X\rightarrow
%TCIMACRO{\U{211d} }%
%BeginExpansion
\mathbb{R}
%EndExpansion
^{n}$ be a lower semicontinuous function bounded from below and satisfying
condition (H). Then, for given $\ \varepsilon ,\delta >0$ and $\ x_{0}\in X$
with%
\begin{equation}
f\left( x_{0}\right) \leq f\left( x\right) +\varepsilon e\ \ \ \text{for all
}x\in X,  \label{ci}
\end{equation}%
there exists a sequence $\left( x_{k}\right) \subset X$ and $x^{\ast }\in X$
such that $x_{k}\rightarrow x^{\ast }\ \ \ $as \ $k\rightarrow \infty ,$%
\begin{equation}
f\left( x^{\ast }\right) \leq f\left( x_{0}\right) ,  \label{s1}
\end{equation}%
\begin{equation}
d\left( x^{\ast },x_{0}\right) \leq \delta e,  \label{s2}
\end{equation}%
\begin{equation*}
f\left( x^{\ast }\right) +\frac{\varepsilon }{\delta }d\left( x^{\ast
},x_{k}\right) \geq f\left( x\right) +\frac{\varepsilon }{\delta }d\left(
x,x_{k}\right) \ \ \text{for all }k\geq 0\ \ \text{\ implies \ }x=x^{\ast }.
\end{equation*}%
Moreover,%
\begin{equation*}
f\left( x^{\ast }\right) \geq f\left( x\right) +\frac{\varepsilon }{\delta }%
Bd\left( x^{\ast },x\right) +\frac{\varepsilon }{\delta }\left( B-I\right)
d\left( x^{\ast },x_{k}\right) \ \ \text{for all }k\geq 0\ \ \ \text{implies
\ }x=x^{\ast }.
\end{equation*}
\end{Theorem}

\begin{proof}
We apply the weak form of Ekeland's variational principle, Theorem \ref{ep},
to the vector $B$-metric $\frac{\varepsilon }{\delta }d.$ From (\ref{c1}),
we immediately obtain (\ref{s1}), while from $x^{\ast }\in F\left(
x_{0}\right) $ and (\ref{ci}), we deduce%
\begin{equation*}
\frac{\varepsilon }{\delta }d\left( x^{\ast },x_{0}\right) \leq f\left(
x_{0}\right) -f\left( x^{\ast }\right) \leq \varepsilon e,
\end{equation*}%
whence (\ref{s2}). The remaining conclusions follow directly.
\end{proof}

A consequence of the weak form of Ekeland's variational principle is the
following version of Caristi's fixed point theorem (see \cite{caristi}) in
vector $B$-metric spaces.

\begin{Theorem}
Let $\left( X,d\right) $ be a complete vector $B$-metric space such that the
$B$-metric $d$ is continuous, and let $f:X\rightarrow
%TCIMACRO{\U{211d} }%
%BeginExpansion
\mathbb{R}
%EndExpansion
^{n}$ be a lower semicontinuous function bounded from below and satisfying
condition (H). Assume that for an operator $N:X\rightarrow X,$ the following
conditions are satisfied:
\begin{equation}
d\left( N\left( x\right) ,y\right) \leq d\left( x,y\right) +Bd\left( N\left(
x\right) ,x\right) ,\ \ \ x,y\in X  \label{cc1}
\end{equation}%
and%
\begin{equation}
Bd\left( N\left( x\right) ,x\right) \leq f\left( x\right) -f\left( N\left(
x\right) \right) ,\ \ \ x\in X.  \label{cc2}
\end{equation}%
Then, $N$ has at least one fixed point.
\end{Theorem}

\begin{proof}
Assume that $N$ has no fixed points. Then, applying Ekeland's variational
principle to $f$ (Theorem \ref{ep}), from (\ref{c2}), one has
\begin{equation*}
f\left( x^{\ast }\right) +d\left( x^{\ast },x_{k}\right) \ngeqslant f\left(
N\left( x^{\ast }\right) \right) +d\left( N\left( x^{\ast }\right)
,x_{k}\right)
\end{equation*}%
for some $k.$ Therefore, there is an index $i$ with%
\begin{equation*}
f_{i}\left( x^{\ast }\right) +d_{i}\left( x^{\ast },x_{k}\right)
<f_{i}\left( N\left( x^{\ast }\right) \right) +d_{i}\left( N\left( x^{\ast
}\right) ,x_{k}\right) .
\end{equation*}%
Using (\ref{cc2}) gives%
\begin{equation*}
\left( Bd\left( N\left( x^{\ast }\right) ,x^{\ast }\right) \right) _{i}\leq
f_{i}\left( x^{\ast }\right) -f_{i}\left( N\left( x^{\ast }\right) \right)
<d_{i}\left( N\left( x^{\ast }\right) ,x_{k}\right) -d_{i}\left( x^{\ast
},x_{k}\right) ,
\end{equation*}%
that is%
\begin{equation*}
d_{i}\left( x^{\ast },x_{k}\right) +\left( Bd\left( N\left( x^{\ast }\right)
,x^{\ast }\right) \right) _{i}<d_{i}\left( N\left( x^{\ast }\right)
,x_{k}\right) ,
\end{equation*}%
which contradicts (\ref{cc1}). Consequently, $N$ has a fixed point.
\end{proof}

\subsection{New versions of the Ekeland variational principle in $b$-metric
spaces}

We emphasize that in the scalar case, that is, when $n=1,\ B=b\geq 1$ and $%
d=\rho $ is a $b$-metric, our theorems from the previous subsection offer
more natural versions in $b$-metric spaces to the classical results, as
follows.

\begin{Theorem}[Weak Ekeland variational principle in $b$-metric spaces]
\label{ep0}Let $\left( X,\rho \right) $ be a complete $b$-metric space ($%
b\geq 1$) such that the $b$-metric $\rho $ is continuous, and let $%
f:X\rightarrow
%TCIMACRO{\U{211d} }%
%BeginExpansion
\mathbb{R}
%EndExpansion
$ be a lower semicontinuous function bounded from below. Then, for given $%
x_{0}\in X,$ there exists a sequence $\left( x_{k}\right) \subset X$ and $%
x^{\ast }\in X$ such that $x_{k}\rightarrow x^{\ast }\ \ \ $as \ $%
k\rightarrow \infty ,$%
\begin{equation*}
f\left( x^{\ast }\right) \leq f\left( x_{0}\right) -\rho \left( x^{\ast
},x_{0}\right) ,
\end{equation*}%
and for each $x\in X,$ $x\neq x^{\ast },$ there exists an index $k=k\left(
x\right) $ with
\begin{equation*}
f\left( x^{\ast }\right) +\rho \left( x^{\ast },x_{k}\right) <f\left(
x\right) +\rho \left( x,x_{k}\right) .
\end{equation*}%
Moreover, for each $x\in X,$ $x\neq x^{\ast },$ there exists an index $%
k=k\left( x\right) $ with%
\begin{equation}
f\left( x^{\ast }\right) <f\left( x\right) +b\rho \left( x^{\ast },x\right)
+\left( b-1\right) \rho \left( x^{\ast },x_{k}\right) .  \label{E}
\end{equation}
\end{Theorem}

\begin{Theorem}[Strong Ekeland variational principle in $b$-metric spaces]
Let $\left( X,\rho \right) $ be a complete $b$-metric space ($b\geq 1$) such
that the $b$-metric $\rho $ is continuous, and let $f:X\rightarrow
%TCIMACRO{\U{211d} }%
%BeginExpansion
\mathbb{R}
%EndExpansion
$ be a lower semicontinuous function bounded from below. Then, for given $\
\varepsilon ,\delta >0$ and $\ x_{0}\in X$ with%
\begin{equation*}
f\left( x_{0}\right) \leq \inf_{x\in X}f\left( x\right) +\varepsilon ,
\end{equation*}%
there exists a sequence $\left( x_{k}\right) \subset X$ and $x^{\ast }\in X$
such that $x_{k}\rightarrow x^{\ast }\ \ \ $as \ $k\rightarrow \infty ,$%
\begin{equation*}
f\left( x^{\ast }\right) \leq f\left( x_{0}\right) ,
\end{equation*}%
\begin{equation*}
\rho \left( x^{\ast },x_{0}\right) \leq \delta ,
\end{equation*}%
and for each $x\in X,$ $x\neq x^{\ast },$ there exists an index $k=k\left(
x\right) $ with%
\begin{equation*}
f\left( x^{\ast }\right) +\frac{\varepsilon }{\delta }\rho \left( x^{\ast
},x_{k}\right) <f\left( x\right) +\frac{\varepsilon }{\delta }\rho \left(
x,x_{k}\right) .
\end{equation*}%
Moreover, for each $x\in X,$ $x\neq x^{\ast },$ there exists an index $%
k=k\left( x\right) $ with%
\begin{equation}
f\left( x^{\ast }\right) <f\left( x\right) +b\frac{\varepsilon }{\delta }%
\rho \left( x^{\ast },x\right) +\left( b-1\right) \frac{\varepsilon }{\delta
}\rho \left( x^{\ast },x_{k}\right) .  \label{Es}
\end{equation}
\end{Theorem}

\begin{Theorem}[Caristi fixed point theorem in $b$-metric spaces]
Let $\left( X,\rho \right) $ be a complete $b$-metric space ($b\geq 1$) such
that the $b$-metric $\rho $ is continuous, and let $f:X\rightarrow
%TCIMACRO{\U{211d} }%
%BeginExpansion
\mathbb{R}
%EndExpansion
$ be a lower semicontinuous function bounded from below. If for an operator $%
N:X\rightarrow X,$ one has%
\begin{equation}
\rho \left( N\left( x\right) ,y\right) \leq \rho \left( x,y\right) +b\rho
\left( N\left( x\right) ,x\right) ,\ \ \ x,y\in X  \label{A}
\end{equation}%
and%
\begin{equation}
b\rho \left( N\left( x\right) ,x\right) \leq f\left( x\right) -f\left(
N\left( x\right) \right) ,\ \ \ x\in X,  \label{C}
\end{equation}%
then $N$ has at least one fixed point.
\end{Theorem}

The last three results reduce to the classical ones in ordinary metric
spaces, i.e., if $b=1.$ Thus, (\ref{E}) reduces to%
\begin{equation*}
f\left( x^{\ast }\right) <f\left( x\right) +\rho \left( x^{\ast },x\right)
,\ \ \ \ x\neq x^{\ast };
\end{equation*}%
(\ref{Es}) reduces to%
\begin{equation*}
f\left( x^{\ast }\right) <f\left( x\right) +\frac{\varepsilon }{\delta }\rho
\left( x^{\ast },x\right) ,\ \ \ \ x\neq x^{\ast };
\end{equation*}%
assumption (\ref{A}) trivially holds, while (\ref{C}) becomes the classical
Caristi's inequality%
\begin{equation*}
\rho \left( N\left( x\right) ,x\right) \leq f\left( x\right) -f\left(
N\left( x\right) \right) ,\ \ \ x\in X.
\end{equation*}

\section{Conclusion and further research}

In this paper, we introduced the concept of a vector $B$-metric space.
Several fixed-point theorems, analogous to those in scalar $b$-metric spaces
as well as their classical counterparts, were presented. Additionally, we
discussed some stability results. Finally, we provided a variant of
Ekeland's variational principle alongside a version of Caristi's theorem. It
remains an open question whether the assumption that $B^{-1}-A$ or $I-BA$ is
inverse-positive can be omitted in Theorems \ref{Avramescu}, \ref%
{stabilitate1} and \ref{stabilitate2}. Additionally, one may explore a
variant of Ekeland's variational principle where Caristi's theorem holds
without requiring the additional assumption \eqref{cc1}. Lastly, it would be interesting to study the case where the matrix $B$ is neither positive nor inverse-positive; for instance, when it has positive diagonal elements but contains both positive and negative entries elsewhere.

\section{Aknowledgements}
The authors wish to mention that the notion of a vector $B$-metric space was
suggested by Professor Ioan A. Rus in the Seminar of Nonlinear Operators and
Differential Equations at Babe\c{s}-Bolyai University.


\begin{thebibliography}{999}

\bibitem{hyres}  
Wilson, W. A. On Quasi-Metric Spaces. \textit{Amer. J. Math.}  \textbf{1931}, 675–684.

% \bibitem{bourbaki}  
% Bourbaki, N. \textit{Topologie Generale}; Herman: Paris, France, 1974.

\bibitem{bahtin}  
Bakhtin, I.A. Contracting mapping principle in an almost metric space. \textit{Funktsionalnyi Analiz} \textbf{1989}, \textit{30}, 26–37.

\bibitem{Czerwik1993}
Czerwik, S. Contraction mappings in \( b \)-metric spaces. \textit{Acta Math. Inform. Univ. Ostrav.} \textbf{1993}, \textit{1}, 5–11.

\bibitem{coifman}
Coifman, R.R.; de Guzman, M. Singular integrals and multipliers on homogeneous spaces. \textit{Rev. Un. Mat. Argentina} \textbf{1970/71}, \textit{25}, 137–143.

\bibitem{hyres2}  
Hyers, D.H. A note on linear topological spaces. \textit{Bull. Amer. Math. Soc.} \textbf{1938}, \textit{44}, 76–80.

\bibitem{Bourgin}  
Bourgin, D.G. Linear topological spaces. \textit{Amer. J. Math.} \textbf{1943}, \textit{65}, 637–659.

\bibitem{berinde}  
Berinde, V.; Păcurar, M. The early developments in fixed point theory on \( b \)-metric spaces: A brief survey and some important related aspects. \textit{Carpathian J. Math.} \textbf{2022}, \textit{38}, 523–538.

\bibitem{survey}  
An, T.V.; Van Dung, N.; Kadelburg, Z.; Radenović, S. Various generalizations of metric spaces and fixed point theorems. \textit{RACSAM} \textbf{2015}, \textit{109}(1), 175–198.

% \bibitem{berinde survey}
% Berinde, V.; Choban, M. Generalized distances and their associate metrics: Impact on fixed point theory. \textit{Creat. Math. Inform.} \textbf{2013}, \textit{22}(1), 23–32.

\bibitem{MITROVIC}  
Mitrović, Z.D. Fixed point results in \( b \)-metric spaces. \textit{Fixed Point Theory} \textbf{2019}, \textit{20}, 559–566. https://doi.org/10.24193/fpt-ro.2019.2.36.
\bibitem{bpr}  
Boriceanu, M.; Petruşel, A.; Rus, I.A. Fixed point theorems for some multivalued generalized contractions in \( b \)-metric spaces. \textit{Int. J. Math. Stat.} \textbf{2010}, \textit{6}, 65–76.

\bibitem{fixed point}
Aydi, H.; Czerwik, S. \textit{Modern Discrete Mathematics and Analysis}. Springer: Cham, Switzerland, 2018.
\bibitem{kirk carte}
Kirk, W.; Shahzad, N. \textit{Fixed Point Theory in Distance Spaces}. Springer: Cham, Switzerland, 2014.


\bibitem{Reich2001}
Reich, S.; Zaslavski, A.J. Well-posedness of fixed point problems. \textit{Far East J. Math. Sci.} \textbf{2001}, Special Volume (Functional Analysis and its Applications), Part III, 393--401.


\bibitem{berinde2}  
Berinde, V. Generalized contractions in quasimetric spaces. \textit{Seminar on Fixed Point Theory}, Preprint no. 3, \textbf{1993}, 3–9.
\bibitem{Miculescu}  
Miculescu, R.; Mihail, A. New fixed point theorems for set-valued contractions in \( b \)-metric spaces. \textit{J. Fixed Point Theory Appl.} \textbf{2017}, \textit{19}, 2153–2163.


\bibitem{Suzuki}  
Suzuki, T. Basic inequality on a \( b \)-metric space and its applications. \textit{J. Inequal. Appl.} \textbf{2017}, \textit{2017}, 256.






\bibitem{stability monica} 
Bota, M.-F.; Micula, S. Ulam--Hyers stability via fixed point results for special contractions in \( b \)-metric spaces. \textit{Symmetry} \textbf{2022}, \textit{14}, 2461.

\bibitem{pp}  
Petrușel, A.; Petrușel, G. Graphical contractions and common fixed points in \( b \)-metric spaces. \textit{Arab. J. Math.} \textbf{2023}, \textit{12}, 423–430. https://doi.org/10.1007/s40065-022-00396-8.

\bibitem{bmv}  
Bota, M.; Molnar, A.; Varga, C. On Ekeland’s variational principle in \( b \)-metric spaces. \textit{Fixed Point Theory} \textbf{2011}, \textit{12}, 21–28.

\bibitem{ekeland2}  
Farkas, C; Molnár, A.; Nagy, S. A generalized variational principle in \( b \)-metric spaces. \textit{Le Matematiche} \textbf{2014}, \textit{69}(2), 205–221.


\bibitem{monica}  
Boriceanu, M. Fixed point theory on spaces with vector-valued \( b \)-metrics. \textit{Demonstr. Math.} \textbf{2009}, \textit{42}, 831–841.




\bibitem{p role}
Precup, R. The role of matrices that are convergent to zero in the study of semilinear operator systems. \textit{Math. Comput. Model.} \textbf{2009}, \textit{49}(3), 703–708.

\bibitem{nonnegative}  
Berman, A.; Plemmons, R.J. Nonnegative matrices in the mathematical sciences. Academic Press: New York,  USA, \textbf{1997}.

\bibitem{collatz}
Collatz, L. Aufgaben monotoner Art. \textit{Arch. Math. (Basel)} \textbf{1952}, \textit{3}, 366–376.



\bibitem{cobzas}
Cobzaș, Ș.; Czerwik, S. The completion of generalized \( b \)-metric spaces and fixed points. \textit{Fixed Point Theory} \textbf{2020}, \textit{21}(1), 133–150.


\bibitem{perov1}
Perov, A.I. On the Cauchy problem for a system of ordinary differential equations (Russian). \textit{Priblizhen. Metody Reshen. Differ. Uravn.} \textbf{1964}, \textit{2}, 115–134.

\bibitem{perov2}
Perov, A.I. Generalized principle of contraction mappings (Russian). \textit{Vestn. Voronezh. Gos. Univ., Ser. Fiz. Mat.} \textbf{2005}, \textit{1}, 196–207.


\bibitem{ortega}
Ortega, J.M.; Rheinboldt, W.C. \textit{Iterative Solutions of Nonlinear Equations in Several Variables}. Academic Press: New York, USA, 1970.



\bibitem{avramescu}
Avramescu, C. Asupra unei teoreme de punct fix. \textit{St. Cerc. Mat.} \textbf{1970}, \textit{22}, 215–221.

\bibitem{ekeland}
Ekeland, I. On the variational principle. \textit{J. Math. Anal. Appl.} \textbf{1974}, \textit{47}(2), 324–353.

\bibitem{mawhin}
 Mawhin, J.; Willem, M. \textit{Critical Point Theory And Hamiltonian Systems}. Applied Mathematical
Sciences: Springer, New York, USA, 1989.

\bibitem{f}
De Figueiredo, D.G. \textit{Lectures on the Ekeland Variational Principle with Applications and Detours}. Tata Institute of Fundamental Research: Bombay, India, 1989.

\bibitem{meghea}
Meghea, I. \textit{Ekeland Variational Principle with Generalizations and Variants}. Old City Publishing: Philadelphia, USA, 2009.

\bibitem{caristi}
Caristi, J. Fixed point theorems for mappings satisfying inwardness conditions. \textit{Trans. Amer. Math. Soc.} \textbf{1976}, \textit{215}, 241–251.

\end{thebibliography}
\end{document}